\documentclass[12]{amsart}
\usepackage{amsmath,amssymb,amsthm,color,enumerate,comment,centernot,enumitem,url,cite}
\usepackage{graphicx,relsize,bm}
\usepackage{mathtools}
\usepackage{array}

\usepackage{hyperref}

\makeatletter
\newcommand{\tpmod}[1]{{\@displayfalse\pmod{#1}}}
\makeatother

\newtheorem{thm}{Theorem}[section]
\newtheorem{lemma}[thm]{Lemma}

\newtheorem{cor}[thm]{Corollary}

\theoremstyle{remark}

\theoremstyle{definition}

\newtheorem{rem}[thm]{Remark}

\theoremstyle{THM}

\newcommand{\V}{{\mathcal V}}
\newcommand{\U}{{\mathcal U}}
\newcommand{\T}{{\mathcal T}}
\newcommand{\SSS}{{\mathcal S}}
\newcommand{\RR}{{\mathcal R}}
\newcommand{\FF}{{\mathcal F}}
\newcommand{\Gal}{{\mbox{{\rm{Gal}}}}}
\newcommand{\rad}{{\mbox{{\rm{rad}}}}}

\newcommand{\mmod}[1]{\ \mathrm{mod}\enspace #1}
\newcommand{\Z}{{\mathbb Z}}
\newcommand{\Q}{{\mathbb Q}}
\newcommand{\R}{{\mathbb R}}

\newcommand{\F}{{\mathbb F}}
\newcommand{\abs}[1]{\left|{#1}\right|}

\makeatletter
\@namedef{subjclassname@2020}{%
  \textup{2020} Mathematics Subject Classification}
\makeatother

\def\red#1 {\textcolor{red}{#1 }}
\def\blue#1 {\textcolor{blue}{#1 }}

\numberwithin{equation}{section}

\begin{document}

\title[Monogenic even sextic trinomials]{Monogenic even sextic trinomials\\ and their Galois groups}


\author{Lenny Jones}
\address{Professor Emeritus, Department of Mathematics, Shippensburg University, Shippensburg, Pennsylvania 17257, USA}
\email[Lenny~Jones]{doctorlennyjones@gmail.com}

\date{\today}

\begin{abstract}
    Let $f(x)=x^6+Ax^{2k}+B\in {\mathbb Z}[x]$, with $A\ne 0$ and $k\in \{1,2\}$. We say that $f(x)$ is {\em monogenic} if $f(x)$ is irreducible over ${\mathbb Q}$ and
    $\{1,\theta,\theta^2,\theta^3,\theta^4,\theta^{5}\}$
    is a basis for the ring of integers of ${\mathbb Q}(\theta)$, where $f(\theta)=0$.

    For each value of $k$ and each possible Galois group $G$ of $f(x)$ over ${\mathbb Q}$, we use a theorem of Jakhar, Khanduja and Sangwan to give explicit descriptions of all monogenic trinomials $f(x)$ having Galois group $G$. We also determine when these descriptions provide infinitely many such trinomials, and we investigate when these trinomials generate distinct sextic fields. 
        
    These results extend recent work on monogenic power-compositional sextic trinomials of the form $g(x^3)$ to the situation $g(x^2)$, and thereby complete the characterization, in terms of their Galois groups, of monogenic power-compositional sextic trinomials. 
\end{abstract}


\maketitle
\section{Introduction}\label{Section:Intro}
 Let
 \begin{equation}\label{Eq:f}
 f(x)=x^6+Ax^{2k}+B\in {\mathbb Z}[x], \quad \mbox{where $A\ne 0$ and $k\in \{1,2\}$.}
 \end{equation} If $f(x)$ is irreducible over ${\mathbb Q}$, then we define $f(x)$ to be {\em monogenic} if \[{\mathcal B}=\{1,\theta,\theta^2,\theta^3,\theta^4,\theta^{5}\}\] is a basis
 for the ring of integers of ${\mathbb Q}(\theta)$, where $f(\theta)=0$. Such a basis ${\mathcal B}$ is sometimes referred to in the literature as a {\em power basis}. We say that two such monogenic sextic trinomials $f_1(x)\ne f_2(x)$ having isomorphic Galois groups over $\Q$ are {\em distinct} if $K_1\ne K_2$, where $K_i=\Q(\theta_i)$ with $f_i(\theta_i)=0$.

 The goal of this article is to provide a characterization of the monogenic sextic trinomials $f(x)$ described in \eqref{Eq:f} in terms of their Galois groups over $\Q$, which we denote $\Gal(f)$. More explicitly, for each possible Galois group $G$, 
 we use a theorem of Jakhar, Khanduja and Sangwan \cite{JKS2} to derive conditions that allow us to present an explicit description of all such monogenic sextic trinomials $f(x)$ in \eqref{Eq:f} with $\Gal(f)\simeq G$. 
  We also investigate when these trinomials generate distinct sextic fields.

   We point out that similar research involving the monogenicity and/or Galois groups of trinomials of various degrees has been conducted by many authors  \cite{AJ,AL,BS,HJAA,HJBAMS,HJActa,JonesQuarticsBAMS,JonesRam,JonesNYJM,JonesJAA,JonesAA,JonesRecipQuartics,JonesEvenSextics,JW,MNSU,S1,S2,Voutier}. The results in this article extend previous work  
   on sextic trinomials of the form $g(x^3)=x^6+Ax^3+B$ \cite{HJActa,JonesRam,JonesNYJM} to sextic trinomials of the form $g(x^2)$, and thereby  complete the characterization, in terms of their Galois groups, of monogenic power-compositional sextic trinomials. 
   In particular, in this article, for $f(x)$ as defined in \eqref{Eq:f}, we see that
   \begin{equation}\label{Eq:g}
   f(x)=g(x^2), \quad \mbox{where} \quad g(x)=x^3+Ax^k+B.
   \end{equation}

 We make some additional notational remarks. Let $m,n\in \Z$, with $m\ge 2$. We let the notation ``$n \mmod{m}$" denote the unique integer $z\in \{0,1,2,\ldots,m-1\}$ such that $n\equiv z \pmod{m}$. That is, $n \mmod{m}=z$. We let $\nu_p(n)$ be the $p$-adic valuation of $n$, and for $n>0$, we let $\rad(n)$, called the ``radical" of $n$, be the product of all distinct prime divisors of $n$. We let $C_n$ denote the cyclic group of order $n$,
  $A_n$ denote the alternating group of order $n!/2$ and $S_n$ denote the symmetric group of order $n!$. Furthermore, $S_4^{+}$ denotes the symmetric group $S_4$ as a subgroup of $S_6$ embedded in $A_6$, and $S_4^{-}$ denotes the group $S_4$ as a subgroup of $S_6$ that is not contained in $A_6$.

  Using the familiar group names as previously described, we provide in Table \ref{T1} the eight possible Galois groups over $\Q$ for a general irreducible even sextic polynomial \cite[Table 2]{AJ}. For the convenience of the reader, we include the ``T"-notation, 6Tn, \cite{BM} for these groups in Table \ref{T1}, since Maple also gives the T-notation when computing the Galois group.
     \begin{table}[h]
 \begin{center}
\begin{tabular}{c|cccccccc}
  Familiar Name & $C_6$ & $S_3$ & $C_2\times S_3$ & $A_4$ & $C_2\times A_4$ & $S_4^{+}$ & $S_4^{-}$ & $C_2\times S_4$\\ [.25em]
 T-notation & 6T1 & 6T2 & 6T3 & 6T4 & 6T6 & 6T7 & 6T8 & 6T11 \\[.25em]
 \end{tabular}
\end{center}
\caption{Possible Galois groups for even sextic polynomials}
 \label{T1}
\end{table}

Our main theorem is:
\begin{thm}\label{Thm:Main}\text{}
Let $A,B,k\in \Z$, with $AB\ne 0$ and $k\in \{1,2\}$. Let 
\[f(x)=x^6+Ax^{2k}+B \quad \mbox{and} \quad  \delta=4A^3+27B^{3-k}.\] Suppose that $f(x)$ is irreducible over $\Q$. Then
$f(x)$ is monogenic with $\Gal(f)\simeq$ 
\begin{enumerate}
  \item \label{M:I1} $C_6$ never occurs; 
  \item \label{M:I2} $S_3$ never occurs;
 \item \label{M:I3} $C_2\times S_3$ if and only if $k=2$ with $(A,B)\in \{(-2,2), \ (2,-2)\}$; 
  \item \label{M:I4} $A_4$ if and only if $(k,A,B)\in \{(1,-3,-1),\ (2,-3,-1)\}$; 
   \item \label{M:I5} $C_2\times A_4$ if and only if 
   \[2\nmid AB,\quad B \ \mbox{is squarefree,}\quad B\ne -1, 
   \quad \rad(\abs{\delta})\mid A, \quad \mbox{such that}\] 
   \begin{enumerate}
     \item $k=1$  with $f(x)\in \FF_1$, where 
     \[\FF_1=\{f(x): 3\nmid A \ \mbox{or} \ (A \mmod{9}, \ B\mmod{9})\in \{(6,1),(6,4),(6,5),(6,8)\}\},\]
      \item $k=2$ with $(A,B)\in \{(-3,1),(-3,3),(3,-3)\}$;\\     
   \end{enumerate}
      \item \label{M:I6} $S_4^{+}$ if and only if 
    $f(x)\in \FF_2$, where 
     \[\FF_2=\{f(x): B=-1, \  A\ne (-1)^k3, \ 4\nmid A, \ 9\nmid A, \ \delta/3^{\nu_3(\delta)} \ \mbox{is squarefree}\};\]
     \item \label{M:I7} $S_4^{-}$ if and only if 
     \begin{enumerate}
     \item $k=1$  with $(A,B)=(-9,-6)$,
     \item $k=2$ with  $f(x)\in \FF_3$, where 
      \[\FF_3=\{f(x): 3\mid A, \ 4\nmid A, \ B\ne -1, \ B=3-4(A/3)^3 \ \mbox{is squarefree}\};\]
   \end{enumerate}
   \item \label{M:I8} $C_2\times S_4$ if and only if 
   \begin{gather*}
   B \ \mbox{and} \quad \delta/\prod_{p\mid 2AB}p^{\nu_p(\delta)} \quad \mbox{are squarefree,}  \quad B\ne -1,\\ 
   \quad -B^{k-1}\delta \quad \mbox{is not a square,}\\
   (A \mmod{4}, \ B\mmod{4})=(3,2) \quad \mbox{if} \ 2\nmid A \ \mbox{and} \ 2\mid B,\\
    (A \mmod{4}, \ B\mmod{4})\in \{(0,1), \ (2,3)\} \quad \mbox{if} \ 2\mid A \ \mbox{and} \ 2\nmid B,
   \end{gather*}  
   and additionally, when
   \begin{enumerate}
   \item $k=1$
   \begin{gather*}
    (A \mmod{9}, \ B\mmod{9})\in \RR_1 \quad \mbox{if} \ 3\mid A \ \mbox{and} \ 3\nmid B;
   \end{gather*}  
   \item $k=2$
   \begin{gather*}
    (A,B)\not \in \{(-2,2), (2,-2)\}, \quad f(x)\not \in \FF_3 \quad \mbox{and}\\
    (A \mmod{9}, \ B\mmod{9})\in \RR_2 \quad \mbox{if} \ 3\mid A \ \mbox{and} \ 3\nmid B,
   \end{gather*}  
   \end{enumerate}
  $\begin{array}{cl}
   \mbox{where} \ \RR_k=&\hspace*{-.1in}\{(0,2),\ (0,4),\ (0,5),\ (0,7),\\
    &\qquad (3,1),\ (3,2),\ (3,8),\ (3,10-3k),\\
     & \qquad \qquad (6,1),\ (6,5),\ (6,8),\ (6,3k+1)\};
     \end{array}$
    \end{enumerate} 
\end{thm}

Observe that the conditions given in Theorem \ref{Thm:General} for a monogenic trinomial $f(x)$ with $\Gal(f)\simeq C_2\times S_4$ do not include a parametric description of all such monogenic trinomials for either value of $k\in \{1,2\}$. However, in the proof of the following corollary, we 
construct infinite parametric families from these conditions, such that the members of these families generate distinct sextic fields. Furthermore, we show that each of the parametric families $\FF_1$, $\FF_2$ and $\FF_3$ given in Theorem \ref{Thm:General}, for the respective Galois groups $C_2\times A_4$, $S_4^{+}$ and $S_4^{-}$, contains an infinite subfamily such that all elements of the subfamily generate distinct sextic fields. More formally, we state the following corollary of Theorem \ref{Thm:Main}. 
\begin{cor}\label{Cor:Main} \text{}
\begin{enumerate}
  \item \label{I1:Cor} The families $\FF_1$, $\FF_2$ and $\FF_3$ of monogenic trinomials given in Theorem \ref{Thm:Main} are infinite, and each of these families contains at least one infinite single-parameter subfamily such that all elements of the subfamily generate distinct sextic fields. 
  \item \label{I2:Cor} For each value of $k\in \{1,2\}$, there exists an infinite single-parameter family of trinomials $f(x)$ such that $f(x)$ is monogenic with $\Gal(f)\simeq C_2\times S_4$, and no two elements in a family generate the same sextic field.      
\end{enumerate}
\end{cor}

\section{Preliminaries}\label{Section:Prelim}
 Throughout this article, with $k\in \{1,2\}$ and $AB\ne 0$, we let
  \begin{align}\label{Basic defs}
  \begin{split}
  f(x)&=x^6+Ax^{2k}+B,\\
  g(x)&=x^3+Ax^k+B,\\
  h(x)&=x^6+(-1)^kAB^{k-1}x^{6-2k}-B^2,\\ 
  \widehat{h}(x)&=x^3+(-1)^kAB^{k-1}x^{3-k}-B^2\ \mbox{and}\\
  \delta&=4A^3+27B^{3-k}.
  \end{split}
   \end{align} Note, from \eqref{Basic defs}, that $f(x)=g(x^2)$ and $h(x)=\widehat{h}(x^2)$.

 The next two lemmas provide some nontrivial information that exists among
  $f(x)$,  $\widehat{h}(x)$ and $h(x)$ concerning their irreducibility. 
\begin{lemma}\label{Lem:widehat(h)}
If $f(x)$ is irreducible over $\Q$, then $\widehat{h}(x)$ is irreducible over $\Q$.
\end{lemma}
\begin{proof}
 By way of contradiction, assume that $\widehat{h}(x)$ is reducible. Then $\widehat{h}(r)=0$ for some nonzero integer $r$. If $k=1$, then
     $r^2(r-A)=B^2$, which implies that $r-A=s^2$ for some nonzero integer $s$. Thus, $B=\pm rs$, and
     \[f(x)=x^6+(r-s^2)x^2\pm rs=(x^2\pm s)(x^4\mp sx^2+r),\]
     which contradicts the fact that $f(x)$ is irreducible. If $k=2$, then
     \begin{equation}\label{Eq:k=2 hh}
     r(r^2+AB)=B^2.
      \end{equation} Let $z=\gcd(r,B)$. Then we can write $r=z\ell$ and $B=zs$ for some integers $\ell$ and $s$  with $\gcd(\ell,s)=1$. Thus, it follows from \eqref{Eq:k=2 hh} that
      \[\ell(z\ell^2+As)=s^2,\] which implies that $\ell\mid s^2$. Hence, $\ell=\pm 1$ since $\ell$ and $s$ are coprime. Therefore, $r\mid B$ and $B=\pm rs$. Then, from \eqref{Eq:k=2 hh}, we see that $A=(s^2-r)/(\pm s)$, which implies that $s\mid r$. Consequently,
      \[f(x)=x^6\pm (s-r/s)x^4\pm rs=(x^2\pm s)(x^4\mp (r/s)x^2+r),\] which again contradicts the fact that $f(x)$ is irreducible.
      Hence, we have established that $\widehat{h}(x)$ is irreducible in either case of $k\in \{1,2\}$.
\end{proof}

 The next lemma follows from  \cite[Theorem 3.3 and Corollary 3.4]{HJMS}.
 \begin{lemma}\label{Lem:HJSextic}
   Suppose that $\widehat{h}(x)$ is irreducible over $\Q$. If $h(x)$ is reducible over $\Q$, then
   \begin{equation}\label{M}
      M(x):=x^4+2(k-2)Ax^2-8Bx+A^{3-k}(-4B)^{k-1}
      \end{equation}
        has a unique integer zero $\mu$, and $h(x)$ 
      factors into irreducibles over $\Z$ as
   \[\left(x^3+\mu x^2+\left(\frac{\mu^2+(k-2)A}{2}\right)x+B\right)\left(x^3-\mu x^2+\left(\frac{\mu^2+(k-2)A}{2}\right)x
   -B\right).\]
 \end{lemma}

  The following theorem, due to Swan \cite{Swan}, gives the formula for the discriminant of an arbitrary trinomial.
\begin{thm}\label{Thm:Swan}
Let ${\mathfrak F}(x)=x^n+Ax^m+B\in \Z[x]$, where $0<m<n$, and let $d=\gcd(n,m)$. Then
\[
\Delta({\mathfrak F})=(-1)^{n(n-1)/2}B^{m-1}\left(n^{n/d}B^{(n-m)/d}-(-1)^{n/d}(n-m)^{(n-m)/d}m^{m/d}A^{n/d}\right)^d.
\]
\end{thm}
\noindent Applying Theorem \ref{Thm:Swan} to $f(x)$ and $g(x)$ yields
\begin{cor}\label{Cor:Swan}
  \[\Delta(f)=-64B^{2k-1}\delta^2 \quad \mbox{and} \quad \Delta(g)=-B^{k-1}\delta.\]
\end{cor}

The next theorem is an adaption of the work in \cite{AJ} to the particular trinomials $f(x)$ in \eqref{Basic defs}, and gives  precise criteria to determine $\Gal(f)$.
 \begin{thm}\label{Thm:AJTri}
 Suppose that $f(x)$ is irreducible over $\Q$. Then, with $\Delta(g)$ as given in Corollary \ref{Cor:Swan}, we have
  \[\Gal(f)\simeq \left\{\begin{array}{cl}
  C_6 & \mbox{if and only if $-B$ is not a square, $\Delta(g)$ is a square}\\
   & \mbox{and $h(x)$ is reducible;}\\[.5em]
  S_3 & \mbox{if and only if neither $-B$ nor $\Delta(g)$ is a square,}\\
  & \mbox{$-B\Delta(g)$ is a square and $h(x)$ is reducible;}\\[.5em]
  C_2\times S_3 & \mbox{if and only if neither $-B$ nor $\Delta(g)$ nor $-B\Delta(g)$}\\
  & \mbox{is a square and $h(x)$ is reducible;}\\[.5em]
  A_4 & \mbox{if and only if $-B$ and $\Delta(g)$ are squares}\\
  & \mbox{and $h(x)$ is irreducible;}\\[.5em]
  C_2\times A_4 & \mbox{if and only if $-B$ is not a square, $\Delta(g)$ is a square}\\
  & \mbox{and $h(x)$ is irreducible;}\\[.5em]
  S_4^{+} & \mbox{if and only if $-B$ is a square, $\Delta(g)$ is not a square}\\
  & \mbox{and $h(x)$ is irreducible;}\\[.5em]
  S_4^{-} & \mbox{if and only if neither $-B$ nor $\Delta(g)$ is a square,}\\
  & \mbox{$-B\Delta(g)$ is a square and $h(x)$ is irreducible;}\\[.5em]
  C_2\times S_4 & \mbox{if and only if neither $-B$ nor $\Delta(g)$ nor $-B\Delta(g)$}\\
  & \mbox{is a square and $h(x)$ is irreducible.}\\[.5em]
  \end{array}\right.\]
\end{thm}

We now present some basic information concerning the monogenicity of a polynomial $\mathfrak{F}(x)$.
Suppose that $\mathfrak{F}(x)\in \Z[x]$ is monic and irreducible over $\Q$.
Let $K=\Q(\theta)$ with ring of integers $\Z_K$, where $\mathfrak{F}(\theta)=0$. Then, we have \cite{Cohen}
\begin{equation} \label{Eq:Dis-Dis}
\Delta(\mathfrak{F})=\left[\Z_K:\Z[\theta]\right]^2\Delta(K),
\end{equation}
where $\Delta(\mathfrak{F})$ and $\Delta(K)$ denote the discriminants over $\Q$, respectively, of $\mathfrak{F}(x)$ and the number field $K$.
Thus, from \eqref{Eq:Dis-Dis}, 
\begin{equation}\label{MonoCondition}
\mathfrak{F}(x) \ \mbox{is monogenic if and only if} \ \Delta(\mathfrak{F})=\Delta(K), \ \mbox{or equivalently},\  \Z_K=\Z[\theta].
\end{equation}
We then have from \eqref{MonoCondition} a sufficient condition for when two monogenic polynomials of the same degree are distinct.
\begin{cor}\label{Cor:distinct}
  Let $\mathfrak{F}_1(x)\ne \mathfrak{F}_2(x)$ be two monogenic polynomials such that $\deg(\mathfrak{F}_1)=\deg(\mathfrak{F}_2)$. Let $K_i=\Q(\theta_i)$, where $\mathfrak{F}_i(\theta_i)=0$. If $\Delta(\mathfrak{F}_1)\ne \Delta(\mathfrak{F}_2)$, then $K_1\ne K_2$ so that 
  $\mathfrak{F}_1(x)$ and $\mathfrak{F}_2(x)$ are distinct.
\end{cor}

The next theorem, due to Jakhar, Khanduja and Sangwan \cite{JKS2}, gives necessary and sufficient conditions for an arbitrary irreducible trinomial to be monogenic. 
\begin{thm}\label{Thm:JKS2} 
Let $n,m\in \Z$ with $n>m\ge 1$.
Let $K=\Q(\theta)$ be an algebraic number field with $\theta\in \Z_K$, the ring of integers of $K$, having minimal polynomial $\mathfrak{F}(x)=x^{n}+Ax^m+B$ over $\Q$, where $\gcd(m,n)=d_0$, $m=m_1d_0$ and $n=n_1d_0$. A prime factor $p$ of $\Delta(\mathfrak{F})$ does not divide $\left[\Z_K:\Z[\theta]\right]$ if and only if $p$ satisfies one of the following conditions:
 \begin{enumerate}[label=(\roman*), font=\normalfont]
  \item \label{JKS:C1} when $p\mid A$ and $p\mid B$, then $p^2\nmid B$;
  \item \label{JKS:C2} when $p\mid A$ and $p\nmid B$, then
  \[\mbox{either } \quad p\mid a_2 \mbox{ and } p\nmid b_1 \quad \mbox{ or } \quad p\nmid a_2\left((-B)^{m_1}a_2^{n_1}+\left(-b_1\right)^{n_1}\right),\]
  where $a_2=A/p$ and $b_1=\frac{B+(-B)^{p^j}}{p}$, such that $p^j\mid\mid n$ with $j\ge 1$;
  \item \label{JKS:C3} when $p\nmid A$ and $p\mid B$, then
  \[\qquad \quad \mbox{either}\quad p\mid a_1 \mbox{ and } p\nmid b_2  \quad\mbox{or}\quad  p\nmid a_1b_2^{m-1}\left((-A)^{m_1}a_1^{n_1-m_1}-\left(-b_2\right)^{n_1-m_1}\right),\]
  where $a_1=\frac{A+(-A)^{p^l}}{p}$, such that $p^l\mid\mid (n-m)$ with $l\ge 0$, and $b_2=B/p$;
  \item \label{JKS:C4} when $p\nmid AB$ and $p\mid m$ with $n=s^{\prime}p^r$, $m=sp^r$, $p\nmid \gcd\left(s^{\prime},s\right)$, then the polynomials
   \begin{equation*}
     G(x):=x^{s^{\prime}}+Ax^s+B \quad \mbox{and}\quad H(x):=\dfrac{Ax^{sp^r}+B+\left(-Ax^s-B\right)^{p^r}}{p}
   \end{equation*}
   are coprime modulo $p$;
         \item \label{JKS:C5} when $p\nmid ABm$, then
     \[p^2\nmid \left(B^{n_1-m_1}n_1^{n_1}-(-1)^{m_1}A^{n_1}m_1^{m_1}(m_1-n_1)^{n_1-m_1}\right).\]
   \end{enumerate}
\end{thm}

The next theorem is an application of Theorem \ref{Thm:JKS2} to the trinomials $f(x)$ in \eqref{Basic defs}.
\begin{thm}\label{Thm:General}
Let $k\in \{1,2\}$, and suppose that $f(x)=x^6+Ax^{2k}+B$ is irreducible over $\Q$. Then
\begin{equation*}\label{Eq:Delf}
\Delta(f)=-64B^{2k-1}\delta^2
\end{equation*} by Corollary \ref{Cor:Swan}. Let $K=\Q(\theta)$, where $f(\theta)=0$, and let $\Z_K$ denote the ring of integers of $K$. A prime factor $p$ of $\Delta(f)$ does not divide $\left[\Z_K:\Z[\theta]\right]$ if and only if $p$ satisfies one of the following conditions:
\begin{enumerate}
  \item \label{C1} when $p\mid A$ and $p\mid B$, then $p^2\nmid B$;
  \item \label{C2} when $p\mid A$ and $p\nmid B$, then $p\in \{2,3\}$ such that 
    \begin{align*}
    (A \mmod{4},\ B \mmod{4})& \in \{(0,1),\ (2,3)\} \quad \mbox{if $p=2$, and}\\
   (A \mmod{9},\ B \mmod{9}) & \in \RR_k\quad \mbox{if $p=3$,}
    \end{align*}
    $\begin{array}{cl}
   \mbox{where} \ \RR_k&\hspace*{-.1in}=\{(0,2),\ (0,4),\ (0,5),\ (0,7),\\
    &\qquad (3,1),\ (3,2),\ (3,8),\ (3,10-3k),\\
     & \qquad \qquad (6,1),\ (6,5),\ (6,8),\ (6,3k+1)\};
     \end{array}$
    \item \label{C3} when $p\nmid A$ and $p\mid B$, then $p\nmid \left[\Z_K:\Z[\theta]\right]$ if $p\ne 2$ and 
    \begin{align*}
  (A \mmod{4},\ B \mmod{4}) & =(3,\ 2) \quad \mbox{if $p=2$};
  \end{align*}
  \item \label{C4} when $2\nmid AB$, then $2\nmid \left[\Z_K:\Z[\theta]\right]$;
  \item \label{C5} when $p\nmid 2AB$, then $p^2\nmid \delta$. 
   \end{enumerate}
\end{thm}
\begin{proof}
Let $p$ be a prime divisor of $\Delta(f)$.
Note that conditions \eqref{C1} and \eqref{C5} remain the same as conditions \ref{JKS:C1} and \ref{JKS:C5} of Theorem \ref{Thm:JKS2}.

For condition \eqref{C2}, we 
suppose that $p\mid A$ and $p\nmid B$.
Then, from condition \ref{JKS:C2} of Theorem \ref{Thm:JKS2}, we have that
  $p^j\mid \mid 6$ with $j\ge 1$. Hence, $p\in \{2,3\}$ with $j=1$ for each of these primes, $a_2=A/p$ and
 \[b_1=\left\{\begin{array}{cl}
   \frac{B+B^2}{2} & \mbox{if $p=2$}\\[.5em]
   \frac{B-B^3}{3} & \mbox{if $p=3$.}
 \end{array}\right.\]

 Observe that, in Theorem \ref{Thm:JKS2}, we have for each value of $k\in \{1,2\}$ that $n_1=3$ and $m_1=k$.
 Then, it is easy to calculate, in each of the two cases $p\in \{2,3\}$ with $k\in \{1,2\}$, the congruence classes of $A$ and $B$ modulo $p^2$ for which
 \[\mbox{either } \ p\mid a_2 \ \mbox{and} \ p\nmid b_1,\quad \mbox{or } \ p\nmid a_2\left((-B)^{k}a_2^{3}-b_1^{3}\right).\] 

 For condition \eqref{C3}, suppose that $p\nmid A$ and $p\mid B$. Then, from condition \ref{JKS:C3} of Theorem \ref{Thm:JKS2}, we have that $p^l\mid \mid (6-2k)$, which implies that either $\ell=0$ if $p\ge 3$, in which case the first statement under condition \ref{JKS:C3} of Theorem \ref{Thm:JKS2} must be true for monogenicity, or $p=2$ and $p^l=2^{3-k}$. In this latter case, we have
 \[a_1=\frac{A^{2^{3-k}}+A}{2}\equiv \left\{ \begin{array}{cl}
  1 \pmod{2} & \mbox{if $A\equiv 1 \pmod{4}$,}\\[.25em]
  0 \pmod{2} & \mbox{if $A\equiv 3 \pmod{4}$.}
 \end{array}\right.\] Since $m=2k>1$, we observe that $4\nmid B$ so that $2\nmid b_2$; otherwise, condition \ref{JKS:C3} of Theorem \ref{Thm:JKS2} fails. It follows that we must have $(A \mmod{4}, \ B \mmod{4})=(3,2)$.

 Finally, for condition \eqref{C4}, suppose that $p\nmid AB$ and $p\mid m$. Thus, $p=2$ since $p\mid 2k$ with $k\in \{1,2\}$. Hence, the possibilities of $(A\mmod{4},\ B\mmod{4})$ are
 \begin{equation}\label{Eq:AB-Congruence classes mod 4}
   (A\mmod{4},\ B\mmod{4})\in \{(1,1),\ (1,3),\ (3,1), \ (3,3)\}
 \end{equation} since $2\nmid AB$. From condition \ref{JKS:C4} of Theorem \ref{Thm:JKS2}, we see that
 \[G(x)=x^3+Ax^k+B \quad \mbox{and} \quad H(x)=\left(\frac{A^2+A}{2}\right)x^{2k}+ABx^k+\left(\frac{B^2+B}{2}\right).\] Straightforward calculations reveal that $G(x)$ and $H(x)$ are coprime in $\F_2[x]$ for all possibilities of $(A\mmod{4},\ B\mmod{4})$ in \eqref{Eq:AB-Congruence classes mod 4}. Consequently, the prime $p=2$ satisfies condition \eqref{C4} when $2\nmid AB$. That is, $2\nmid \left[\Z_K:\Z[\theta]\right]$ when $2\nmid AB$.
\end{proof}

The following corollary, which follows immediately from conditions \eqref{C1} and \eqref{C3} of Theorem \ref{Thm:General}, will be useful in the proof of Theorem \ref{Thm:Main}. 

\begin{cor}\label{Cor:B squarefree}
  If $f(x)$ is monogenic, then $B$ is squarefree.
\end{cor}

\section{The Proof of Theorem \ref{Thm:Main}}\label{Section:Main3Proof}
We first prove a lemma. 
\begin{lemma}\label{Lem:h reducible} 
Suppose that $f(x)$ is irreducible over $\Q$. If $h(x)$ is reducible over $\Q$, then $\Delta(g)<0$.  
\end{lemma}
\begin{proof}
 Since $f(x)$ is irreducible, then $\widehat{h}(x)$ is irreducible from Lemma \ref{Lem:widehat(h)}. Then,  since $h(x)$ is reducible, it follows from Lemma \ref{Lem:HJSextic} that $M(\mu)=0$ for some $\mu\in \Z$. 
    Calculating $\Delta(M)$, using \cite[Theorem 3.1]{G} when $k=1$, and Theorem \ref{Thm:Swan} when $k=2$, reveals that
      \begin{equation}\label{Delta(M)}
      \Delta(M)=-2^{12}B^{k+1}(4A^3+27B^{3-k})=2^{12}B^2\Delta(g).
      \end{equation}
       Note that $g(x)$ is irreducible, since $f(x)$ is irreducible, so that $\Delta(g)\ne 0$. Assume, by way of contradiction, that $\Delta(g)>0$. Thus, $\Delta(M)>0$, from \eqref{Delta(M)}. Consequently, $M(x)$ has four distinct real zeros since $\mu\in \R$. Note that since $A\ne 0$, it follows from \eqref{M} that $\mu \ne 0$ for either value of $k$, and $\mu \ne -A/2$ if $k=2$. Observing from \eqref{M} that 
       \[\begin{array}{cl}
       \mu^3-2A\mu-8B=-A^2/\mu & \mbox{if $k=1$}\\[.5em]
       \mbox{and}\\[.5em]
       \mu^3-8B=\mu^3A/(A+2\mu) & \mbox{if $k=2$,}
       \end{array}\]
      straightforward calculations show that $M(x)=(x-\mu)Q(x)$, where \[Q(x)=\left\{\begin{array}{cl}
        x^3+\mu x^2+(\mu^2-2A)x-A^2/\mu & \mbox{if $k=1$}\\[.5em]
        x^3+\mu x^2+\mu^2x+\mu^3A/(A+2\mu) & \mbox{if $k=2$.}
      \end{array}\right.\]   Calculating $\Delta(Q)$ using \cite[Theorem 3.1]{G}, or Theorem \ref{Thm:Swan} if $\mu^2=2A$ and $k=1$, yields
       \begin{equation}\label{Delta(Q)}
       \Delta(Q)=\left\{\begin{array}{cl}
        -(3\mu^4-14A\mu^2+27A^2)(-\mu^2+A)^2/\mu^2 & \mbox{if $k=1$}\\[.5em]
        -4\mu^6(3\mu^2-4A\mu+4A^2)/(A+2\mu)^2 & \mbox{if $k=2$.}
      \end{array}\right.
       \end{equation} Since $Q(x)$ has three distinct real zeros, then $\Delta(Q)>0$, which implies from \eqref{Delta(Q)} that
      \[-(3\mu^4-14A\mu^2+27A^2)>0 \ \mbox{if $k=1$, and} \ -(3\mu^2-4A\mu+4A^2)>0 \ \mbox{if $k=2$}.\]  However,
      \[-(3\mu^4-14A\mu^2+27A^2)=-(3(\mu^2-7A/3)^2+32A^2/3)<0\] and
      \[-(3\mu^2-4A\mu+4A^2)=-(3(\mu-2A/3)^2+8A^2/3)<0.\] These contradictions complete the proof that $\Delta(g)<0$.
\end{proof}

The following corollary of Lemma \ref{Lem:h reducible} indicates that the list of possible Galois groups for $f(x)$ is slightly smaller than the list given in Table \ref{T1} for an arbitrary even sextic polynomial.
\begin{cor}\label{Cor:ShortList}
Suppose that $f(x)$ is irreducible over $\Q$.
Then $\Gal(f)\not \simeq C_6$.
\end{cor}
\begin{proof}
 We proceed by way of contradiction, assuming that $\Gal(f)\simeq C_6$. Then, from Theorem \ref{Thm:AJTri}, we have that:
  \begin{equation}\label{C6 conditions}
    -B \ \mbox{is not a square}, \quad \Delta(g) \ \mbox{is a square,} \quad  \mbox{and} \ h(x) \ \mbox{is reducible}.
  \end{equation} Since $h(x)$ is reducible, we have from Lemma \ref{Lem:h reducible} that $\Delta(g)<0$, which contradicts the fact that $\Delta(g)$ is a square.
   \end{proof}
     \begin{rem}
       It is easy to verify that the groups listed in Table \ref{T1}, with the exception of $C_6$,  actually do occur as Galois groups of $f(x)$.
     \end{rem}

\begin{proof}[Proof of Theorem \ref{Thm:Main}]
We address the Galois groups, one at a time, as presented in the statement of the theorem. 
It will sometimes be convenient to handle the cases $k=1$ and $k=2$ separately. Recall that 
 \begin{align}\label{k=1}
 \begin{split}
  f(x) & =x^6+Ax^2+B,\\
  \delta & =4A^3+27B^2,\\
  \Delta(f) & =-2^6B\delta^2 \ \mbox{and}\\
  \Delta(g) & =-\delta,
  \end{split}
\end{align} 
when $k=1$, while 
 \begin{align}\label{k=2}
 \begin{split}
  f(x) & =x^6+Ax^4+B,\\
  \delta & =4A^3+27B,\\
  \Delta(f) & =-2^6B^3\delta^2 \ \mbox{and}\\
  \Delta(g) & =-B\delta,
  \end{split}
\end{align} 
when $k=2$.

 \subsection*{{\bf The Case} $\mathbf{C_6}$}
 By Corollary \ref{Cor:ShortList}, it follows that no monogenic trinomials $f(x)$ exist with $\Gal(f)\simeq C_6$.

\subsection*{{\bf The Case} $\mathbf{S_3}$} Assume, by way of contradiction, that $f(x)$ is monogenic with  $\Gal(f)\simeq S_3$.
By Theorem \ref{Thm:AJTri}, we have:
\begin{equation}\label{AJ conditions S3}
-B \ \mbox{is not a square}, \quad -B\Delta(g) \ \mbox{is a square}, \quad h(x) \ \mbox{is reducible.}
\end{equation}

We first address the case $k=1$ in \eqref{k=1}. 
Suppose that $p$ is a prime divisor of $B$. Then $p\mid \mid B$ by Corollary \ref{Cor:B squarefree}, which implies that $p\mid \delta$ since $-B\Delta(g)=B\delta$ is a square. If $p\not \in \{2,3\}$, then $p^3\mid \mid B\delta$, which contradicts the fact that $B\delta$ is a square. Hence, $p\in \{2,3\}$. Consequently,
\begin{equation}\label{B}
\abs{B}\in \{1,2,3,6\},
\end{equation} since $B$ is squarefree. Since $-B\Delta(g)$ is a square, we can write
\[y^2=-B\Delta(g)=B(4A^3+27B^2),\]
for some $y\in \Z$,
 so that
\begin{equation}\label{k=1 y^2}
  (4By)^2=(4AB)^3+(16)(27)B^5.
\end{equation}
It follows then from \eqref{k=1 y^2} that, for each of the eight values of $B$ in \eqref{B}, we need to determine the integral points $(X,Y)$ on the elliptic curve
\[Y^2=X^3+2^43^3B^5,\] such that $X=4AB$ for some integers $A$ and $Y$ with $AY\ne 0$. 
To do this, we use Sage and we provide the results of these calculations in Table \ref{T2}.
\begin{table}[h]
 \begin{center}
\begin{tabular}{c|cccccccc}
  $B$ & $1$ & $-1$ & $2$ & $-2$ & $3$ & $-3$ & $6$ & $-6$\\ [.25em]
 $A$ & none &  $-3$ & none  & $-5$ & none & none & $-3$ & $-9$
 \end{tabular}
\end{center}
\caption{Possible values for $A$ and $B$}
 \label{T2}
\end{table}\\
From Table \ref{T2}, we get the four trinomials:
\begin{equation}\label{4tri k=1}
x^6-3x^2-1,\quad x^6-5x^2-2,\quad x^6-3x^2+6, \quad x^6-9x^2-6.
\end{equation} Observe that 
\[x^6-5x^2-2=(x^2+2)(x^4-2x^2-1).\] Straightforward methods verify that the remaining three trinomials in \eqref{4tri k=1} are irreducible over $\Q$. Then, a computer algebra system, such as Maple, Sage or Magma, in conjunction with Theorem \ref{Thm:General}, can be used to complete the proof that there exist no even monogenic trinomials $f(x)$ with $\Gal(f)\simeq S_3$ when $k=1$. The results of these computations are summarized in Table \ref{T3}.
\begin{table}[h]
 \begin{center}
\begin{tabular}{c|ccc}
  $f(x)$ & $x^6-3x^2-1$ & $x^6-3x^2+6$ & $x^6-9x^2-6$\\ [.25em]
 $\Gal(f)$ & $A_4$ & $S_3$ & $S_4^{-}$\\ [.25em]
 Monogenic & yes & no & yes
 \end{tabular}
\end{center}
\caption{Results for trinomials when $k=1$}
 \label{T3}
\end{table}

Suppose next that $k=2$ as in \eqref{k=2}. 
Since $-B\Delta(g)=B^2\delta$ is a square, we deduce that $\delta$ must be a square. Let $p$ be a prime divisor of $\delta$. Then, since $f(x)$ is monogenic, it follows from condition \eqref{C5} of Theorem \ref{Thm:General} that $p\mid 2AB$. If $p\ne 3$, then $p^3\mid \mid \delta$ by Corollary \ref{Cor:B squarefree}, which contradicts the fact that $\delta$ is a square. Hence, $p=3$, so that $3\mid AB$ and   
\begin{equation}\label{delta_2}
\delta=4A^3+27B=3^{2z},
\end{equation} for some integer $z\ge 0$. 

We consider first the two cases $z\in \{0,1\}$ in \eqref{delta_2}. 
If $z=0$, then we see that $3\nmid A$. Therefore, $3\mid B$, and we get  that 
$A^3\equiv 61 \pmod{81}$, which is impossible since $61$ is not a cube modulo 81. Suppose next that $z=1$. If $3\mid A$, then it follows from \eqref{delta_2} that
\[3\left(4(A/3)^3+B\right)=1,\] which is impossible. Thus, $3\nmid A$ and $3\mid B$, so that $A^3\equiv 63 \pmod{81}$ from \eqref{delta_2}, which is also impossible since $63$ is not a cube modulo 81. 

We have shown that $z\ge 2$ in \eqref{delta_2}. Therefore, $2z-3$ is a positive integer, and we can rewrite \eqref{delta_2} as
\begin{equation}\label{delta_2 rewrite 2}
B=3^{2z-3}-4(A/3)^3,
\end{equation}  
which implies that $A/3\in \Z$. If $z\ge 3$, then we deduce from \eqref{delta_2 rewrite 2} and Corollary \ref{Cor:B squarefree} that $3\nmid (A/3)$. Thus, $3\nmid B$ and 
\[(A \mmod{9},\ B \mmod{9})\in \{(3,5),\ (6,4)\}.\] Since $\{(3,5),\ (6,4)\} \cap \RR_2=\varnothing$, we conclude from condition \eqref{C3} of Theorem \ref{Thm:General} that $f(x)$ is not monogenic when $z\ge 3$. 

Hence, we only need to examine the situation $z=2$. Then we have 
\begin{equation}\label{z=2}
B=3-4(A/3)^3,
\end{equation} from \eqref{delta_2 rewrite 2}. Since $f(x)$ is monogenic, then $f(x)$ is irreducible over $\Q$, and so  $\widehat{h}(x)$ is irreducible over $\Q$ by Lemma \ref{Lem:widehat(h)}. Since $h(x)$ is reducible over $\Q$ from \eqref{AJ conditions S3}, we have, from Lemma \ref{Lem:HJSextic}, that  
\begin{equation}\label{M k=2}
M(x)=x^4-8Bx-4AB
\end{equation} is reducible over $\Q$. 
Let $p$ be a prime divisor of $B$. Observe from \eqref{z=2} that $p\ne 2$. If $p\nmid A$, then, since $B$ is squarefree, we get the contradiction that $M(x)$ in \eqref{M k=2} is irreducible over $\Q$ because $M(x)$ is Eisenstein with respect to $p$. Thus, $p\mid A$, and we deduce from \eqref{z=2} that $p=3$. Hence, we only need to address the cases $\abs{B}\in \{1,3\}$. From \eqref{z=2}, it is easy to see that cases $B\in \{1,-3\}$ are impossible, while the case $B=3$ yields the contradiction that $A=0$. The only viable case is then $(A,B)=(3,-1)$. However, in this case, it is straightforward to verify that $M(x)=x^4+8x+12$ is irreducible over $\Q$. This final contradiction completes the proof of the nonexistence of even monogenic trinomials $f(x)$ with $\Gal(f)\simeq S_3$.

\subsection*{{\bf The Case} $\mathbf{C_2 \times S_3}$}
Assume that $f(x)$ is monogenic with $\Gal(f)\simeq C_2\times S_3$.
By Theorem \ref{Thm:AJTri}, we have:
\begin{equation}\label{AJ conditions C2 X S3}
\mbox{neither} \ -B, \ \mbox{nor} \ \Delta(g), \ \mbox{nor} \ -B\Delta(g) \ \mbox{is a square,} \ \mbox{and} \ h(x) \ \mbox{is reducible.}
\end{equation}

Suppose first that $k=1$. Since $h(x)$ is reducible, we have from Lemma \ref{Lem:HJSextic} that 
\begin{equation}\label{M(mu) k=1 C2 X S3}
M(\mu)=\mu^4-2A\mu^2-8B\mu+A^2=0,
\end{equation} for some $\mu\in \Z$. Solving \eqref{M(mu) k=1 C2 X S3} for $A$ yields
\begin{equation}\label{Solving for A}
A=\mu^2\pm 2\sqrt{2B\mu},
\end{equation} which implies that $2B\mu$ is a nonzero square, since $A,B\in \Z$ with $AB\ne 0$. Observe from \eqref{M(mu) k=1 C2 X S3} that 
\begin{equation}\label{parity of mu and A}
\mu\equiv A\pmod{2}.
\end{equation} 
We claim that $2\mid \mu$. Suppose, by way of contradiction, that $2\nmid \mu$. Then $2\nmid A$ from \eqref{parity of mu and A}, and $2\mid B$ since $2B\mu$ is a nonzero square. Furthermore,   
\[\mu^2\equiv A^2\equiv 1\pmod{8},\] so that $2(A-1)\equiv 0 \pmod{8}$ from \eqref{M(mu) k=1 C2 X S3}. Hence,  $A\equiv 1 \pmod{4}$, which contradicts the fact that $A\equiv 3 \pmod{4}$ from condition \eqref{C3} of Theorem \ref{Thm:General}, and the claim that $2\mid \mu$ is established. 
Observe from \eqref{Solving for A} that  if $2\mid B$, then $(B/2)\mid \mu$ since $2B\mu$ is a nonzero square and $B$ is squarefree. However, since $2\mid \mu$ and $2\nmid (B/2)$, we actually have that $2(B/2)=B$ divides $\mu$. Clearly, if $2\nmid B$, then $B\mid \mu$. Therefore, $B\mid \mu$, regardless of the parity of $B$.

Let 
\[B=2^{\varepsilon}\prod_{p_i\ge 3}p_i \qquad \mbox{and} \qquad 2^e\mid \mid \mu,\]
where $\varepsilon\in \{0,1\}$, $e\ge 1$ and the $p_i$ are the odd prime divisors of $B$. Thus, since $2B\mu$ is a nonzero square, we have that
\begin{equation}\label{2Bmu}
2B\mu=2^{e+1+\varepsilon}\prod_{p_i\ge 3}p_i^2\left(\frac{\mu}{2^e\prod_{p_i\ge 3}p_i}\right),
\end{equation}
where $2\mid (e+1+\varepsilon)$ and $\mu/(2^e\prod_{p_i\ge 3}p_i)=m^2$ for some nonzero odd integer $m$. Hence, 
\begin{equation}\label{mu}
\mu=2^{e-\varepsilon}m^22^{\varepsilon}\prod_{p_i\ge 3}p_i=2^{2r-1}m^2B,
\end{equation}
 where $r=(e+1-\varepsilon)/2\ge 1$. 
Then 
\begin{equation}\label{A}
  A=2^{r+1}mB\left(2^{3r-3}m^3B\pm 1\right),
\end{equation}
from \eqref{Solving for A} and \eqref{mu}.
An easy calculation shows that 
\begin{equation}\label{newdelta}
  \delta=4A^3+27B^2=B^2\left(2^{3r-1}m^3B\pm 1\right)^2\left(2^{6r-1}m^6B^2\pm 2^{3r}5m^3B+27\right),
\end{equation} where the plus sign holds in \eqref{newdelta} if and only if the plus sign holds in \eqref{A}. Let $d=2^{3r-1}m^3B\pm 1$. Note that $\abs{d}>1$ since $r\ge 1$ and $mB\ne 0$. Let $p$ be a prime divisor of $d$, and observe that $p\nmid 2mB$. Since $f(x)$ is monogenic and $p^2\mid \delta$ from \eqref{newdelta}, it follows from condition \eqref{C5} of Theorem \ref{Thm:General} that $p\mid A$. Hence, 
\[p\mid \left(2^{3r-3}m^3B\pm 1\right),\] from \eqref{A}. Consequently, $p$ divides 
\[d-\left(2^{3r-3}m^3B\pm 1\right)=3\cdot 2^{3r-3}m^3B,\] so that $p=3$. Therefore, 
\begin{equation}\label{d}
 \abs{d}=\abs{2^{3r-1}m^3B\pm 1}=3^z, 
\end{equation} for some integer $z\ge 1$. If $z=1$ in \eqref{d}, we get that 
\[2^{3r-1}m^3B\in \left\{ \begin{array}{cl}
\{-4,2\} &  \mbox{when the plus sign holds in \eqref{d} and \eqref{A}}\\
\{-2,4\} & \mbox{when the minus sign holds in \eqref{d} and \eqref{A}.}
\end{array}\right.\] It follows that $r=1$, so that $m^3B=\pm 1$. Since $-B$ is not a square from \eqref{AJ conditions C2 X S3}, we see that $B\ne -1$. We conclude that $(m,B)=(\pm 1,1)$, where the plus sign (minus sign) holds if and only if the minus sign (plus sign) holds in \eqref{d} and \eqref{A}. Both of these scenarios yield the contradiction that $A=0$. 

Hence, we may assume that $z\ge 2$ in \eqref{d}. Thus,
\[B\mmod{9}=\frac{\mp 1}{2^{3r-1}m^3} \mmod{9}\in \{2,7\},\] which implies from \eqref{A}, regardless of the sign, that 
\begin{equation}\label{A mod 9}
A \mmod{9}=\frac{-3}{2^{2r}m^2}.
\end{equation} Let $S=\{1,4,7\}$, which is the multiplicative subgroup of squares modulo 9. Since $2^{2r} \mmod{9}\in S$, $m^2 \mmod{9}\in S$ and 
\[\frac{-3}{s_1}+\frac{3}{s_2}=-3\left(\frac{s_2-s_1}{s_1s_2}\right)\equiv 0 \pmod{9}\] for any $s_1,s_2\in S$, it follows from \eqref{A mod 9} that $A \mmod{9}=6$. Hence, 
\[(A \mmod{9}, \ B\mmod {9})\in \{(6,2), (6,7)\},\] which yields the contradiction that $f(x)$ is not monogenic by condition \eqref{C2} of Theorem \ref{Thm:General} since neither $(6,2)$ nor $(6,7)$ is an element of $\RR_1$. This completes the proof that there are no monogenic trinomials $f(x)$ with $\Gal(f)\simeq C_2\times S_3$ in the case of $k=1$.

Suppose next that $k=2$. 
Since $h(x)$ is reducible, we have from Lemma \ref{Lem:HJSextic} that 
\begin{equation}\label{M(mu) k=2 C2 X S3}
M(\mu)=\mu^4-8B\mu-4AB=0, 
\end{equation} for some $\mu\in \Z$, or equivalently,
\begin{equation}\label{mu4}
 \mu^4=4B(2\mu+A). 
\end{equation} Note that $\mu\ne 0$ since $AB\ne 0$. Observe from  \eqref{mu4} that $B\mid \mu$ since $B$ is squarefree. Thus, we can write $\mu=mB$, for some nonzero integer $m$.
Then, noting that $2\mid \mu$ from \eqref{mu4}, and solving for $A$ in \eqref{M(mu) k=2 C2 X S3}, we get 
\begin{align}\label{Solving for A k=2}
\nonumber A&=\frac{\mu(\mu^3-8B)}{4B}\\
\nonumber &=\frac{mB(m^3B^2-8)}{4}\\
&=\left(\frac{mB}{2}\right)\left(\frac{m^3B^2}{2}-4\right),
\end{align} 
where each factor in \eqref{Solving for A k=2} is an integer. 
Hence,
\begin{align}\label{newdelta k=2}
 \nonumber \delta&=4A^3+27B\\
 \nonumber &=\frac{B(m^6B^4-20m^3B^2+108)(m^3B^2-2)^2}{16}\\
  &=B\left(\frac{m^6B^4}{4}-5m^3B^2+27\right)\left(\frac{m^3B^2}{2}-1\right)^2,
 \end{align}
where each factor in \eqref{newdelta k=2} is an integer.

Let $d=m^3B^2/2-1$ from \eqref{newdelta k=2}. Since $f(x)$ is irreducible, we know that $d\ne 0$. We claim that $\abs{d}=1$. Suppose, by way of contradiction, that $\abs{d}>1$, and let $p$ be a prime divisor of $d$. Observe that $p\nmid mB$, which also implies that $p\ne 2$ since $2\mid mB$. Since $f(x)$ is monogenic and $p^2\mid \delta$ from \eqref{newdelta k=2}, it follows from condition \eqref{C5} of Theorem \ref{Thm:General} that $p\mid A$. Hence, 
\[p\mid (m^3B^2/2-4),\] from \eqref{Solving for A k=2}. Consequently, $p$ divides 
\[d-\left(\frac{m^3B^2}{2}-4\right)=3,\] so that $p=3$. Therefore, 
\begin{equation}\label{d k=2}
 \abs{d}=\abs{m^3B^2/2-1}=3^z, 
\end{equation} for some integer $z\ge 1$. If $z=1$ in \eqref{d k=2}, then it is easy to see that 
\[(m,B)\in \{(2,\pm 1), (-1,\pm 2)\}.\] The solutions $(m,B)=(2,\pm 1)$ yield the contradiction that $A=0$. The solutions $(m,B)=(-1,\pm 2)$ produce the pairs $(A,B)\in \{(-6,-2),(6,2)\}$, so that 
\[(A \mmod{9},\ B \mmod{9})\in \{(3,7), (6,2)\}.\] However, neither of these pairs is contained in $\RR_2$ from condition \eqref{C2} of Theorem \ref{Thm:General}, which contradicts the fact that $f(x)$ is monogenic. 

Hence, we may assume that $z\ge 2$ in \eqref{d k=2}. Then, $m^3B^2\equiv 2 \pmod{9}$, which implies that 
\begin{equation}\label{mB mod 9 list k=2}
(m \mmod{9}, \ B \mmod{9})\in \{(2, 4), (2, 5), (5, 4), (5, 5), (8, 4), (8, 5)\}.
\end{equation} It follows from \eqref{mB mod 9 list k=2} and \eqref{Solving for A k=2} that
\[(A\mmod {9}, \ B\mmod{9})\in \{(3,5), (6,4)\},\] and neither of these pairs is an element of $\RR_2$ from condition \eqref{C2} of Theorem \ref{Thm:General}, which again contradicts the fact that $f(x)$ is monogenic. Thus, we have established the claim that $\abs{d}=1$, so that $z=0$ in \eqref{d k=2}.  

The only solutions to $\abs{d}=1$ are easily found to be $(m,B)\in \{(1,\pm 2)\}$, and these solutions lead to the pairs $(A,B)\in \{(2,-2), (-2,2)\}$. Using Theorem \ref{Thm:AJTri} and Theorem \ref{Thm:General} (or a computer algebra system), it is straightforward to verify that $f(x)$ is monogenic with $\Gal(f)\simeq C_2\times S_3$, for each of these pairs $(A,B)$, which completes the proof of this case.

\subsection*{{\bf The Case} $\mathbf{A_4}$}
Assume that $f(x)$ is monogenic with $\Gal(f)\simeq A_4$.
By Theorem \ref{Thm:AJTri}, we have:
\begin{equation*}\label{AJ conditions A_4}
-B \quad \mbox{and} \quad \Delta(g) \quad \mbox{are squares,} \quad \mbox{and} \quad h(x)\quad \mbox{is irreducible.}
\end{equation*}
Then, by Corollary \ref{Cor:B squarefree}, it follows that $B=-1$.  Thus, since 
\[\Delta(g)=-B^{k-1}\left(4A^3+27B^{3-k}\right)=(-1)^k\left(4A^3+27(-1)^{3-k}\right)=(-1)^k4A^3-27\] is a square, we need to find the integral points on the elliptic curve
\[E: \quad Y^2=X^3-2^43^3,\] where $X=(-1)^k4A$. Using Sage, we get $X=12$, which yields $A=-3$ when $k=1$, and $A=3$ when $k=2$. It is easy to confirm that \[(k,A,B)\in \{(1,-3,-1), (2,3,-1)\}\] produce monogenic even sextic trinomials $f(x)$ with $\Gal(f)\simeq A_4$.

\subsection*{{\bf The Case} $\mathbf{C_2 \times A_4}$}
Suppose first that $f(x)$ is monogenic with $\Gal(f)\simeq C_2\times A_4$. Then 
\begin{equation}\label{AJ conditions C2 X A4}
-B \ \mbox{is not a square,} \quad \Delta(g) \ \mbox{is a square,} \quad \mbox{and} \quad h(x) \ \mbox{is irreducible,}
\end{equation} by Theorem \ref{Thm:AJTri}. Regardless of the value of $k$, we need to show that
 \begin{equation}\label{Gen conditions C2 X A4}
 2\nmid AB,\quad B\ne -1 \ \mbox{is squarefree}\quad \mbox{and} \quad \rad(\abs{\delta})\mid A.
 \end{equation}

To show that $2\nmid AB$, we proceed by way of contradiction.  Suppose first that $2\mid A$. Since $B$ is squarefree, we have that 
\[B\mmod{32}\not \in \{0,4,8,12,16,20,24,28\}.\] Consequently,    
\[\Delta(g) \mmod{32}\in \{5, 13, 20, 21, 29\},\] regardless of the value of $k$, which contradicts the fact that $\Delta(g)$ is a square, since the squares modulo 32 are $\{0, 1, 4, 9, 16, 17, 25\}$. Hence, $2\nmid A$. Suppose that $2\mid B$. Then, since $f(x)$ is monogenic, we have from condition \eqref{C3} of Theorem \ref{Thm:General} that $(A\mmod{4},B\mmod{4})=(3,2)$. But then 
\[\Delta(g)\equiv \left\{\begin{array}{cl}
8 \pmod{16} & \mbox{if $k=1$};\\
12 \pmod{16} & \mbox{if $k=2$.}
\end{array} \right.\] which contradicts the fact that $\Delta(g)$ is a square. Therefore, $2\nmid B$, and the claim that $2\nmid AB$ is established. 

Since $f(x)$ is monogenic, we have from Corollary \ref{Cor:B squarefree} that $B$ is squarefree. Since $-B$ is not a square, we deduce that $B\ne -1$.

We turn next to showing that $\rad(\abs{\delta})\mid A$.  
Since 
\begin{equation*}\label{Gen delta}
\Delta(g)=-B^{k-1}\delta=-B^{k-1}(4A^3+27B^{3-k})
\end{equation*} is a square, we deduce from condition \eqref{C5} of Theorem \ref{Thm:General} that $\rad(\abs{\delta})\mid 2AB$. Let $p$ be a prime divisor of $\delta$. Since $2\nmid B$, it follows that $p\ne 2$. Therefore, $p\mid A$ or $p\mid B$, with $p\ge 3$. If $p\nmid A$, then $p\mid B$, which yields the contradiction that $p\nmid \delta$. Thus, $p\mid A$, which completes the proof that conditions \eqref{Gen conditions C2 X A4} hold, regardless of the value of $k$.

 Moving forward, we split our analysis into the two cases $k\in \{1,2\}$, as described in \eqref{k=1} and \eqref{k=2}. Consider first the case $k=1$. Suppose that $3\mid \Delta(f)$. If $3\nmid A$,  then $3\mid B$, and we see from condition \eqref{C3} of Theorem \ref{Thm:General} that $3\nmid \left[\Z_K:\Z[\theta]\right]$. Suppose then that $3\mid A$. We claim that $3\nmid B$. Assume, to the contrary, that $3\mid B$, and let 
 \[Z:=4(A/3)^3+B^2\in \Z.\] Then, since $\Delta(g)=-3^3Z$ is a square, we have that $3\mid Z$, which implies that $3\mid (A/3)$. Hence, $3^2\mid \mid Z$ since $3\mid \mid B$, so that $3^5\mid \mid \Delta(g)$, which contradicts the fact that $\Delta(g)$ is a square. Thus, $3\nmid B$, and it is easy to verify that 
\[(A \mmod{9}, \ B \mmod{9})\in \{(6,1),(6,4),(6,5),(6,8)\},\] from condition \eqref{C2} of Theorem \ref{Thm:General}. 

Next, suppose that $p$ is a prime divisor of $\Delta(f)$ with $p\ge 5$. If $p\mid\delta$, then $p\mid A$  since $\rad(\abs{\delta})\mid A$, which implies that $p\mid B$. Since $B$ is squarefree, we see that condition \eqref{C1} of Theorem \ref{Thm:General} is satisfied with no further restrictions on $A$ and $B$.  If $p\mid B$ and $p\nmid \delta$, then $p\nmid A$ so that $p\nmid \left[\Z_K:\Z[\theta]\right]$ from condition \eqref{C3} of Theorem \ref{Thm:General}, which concludes our analysis of the case $k=1$. 

Consider now the case $k=2$. Since $\Delta(g)=-B(4A^3+27B)$ is a square, $2\nmid B$ and $B$ is squarefree, it follows that $B\mid \delta$. Thus, $B\mid A$ and 
\begin{equation}\label{del/B}
-\delta/B=-\left(4A^2(A/B)+27\right)
\end{equation} 
is a square, since $\Delta(g)=-B\delta=B^2(-\delta/B)$
is a square. Let $p$ be a prime divisor of $\delta/B$. Then $p\mid A$, since $\rad(\abs{\delta})\mid A$ from \eqref{Gen conditions C2 X A4}. Thus, $p=3$ from \eqref{del/B} and 
\[-\delta/B=-\left(4A^2(A/B)+27\right)=3^{2n},\]
for some integer $n\ge 0$, or equivalently, 
\begin{equation}\label{equation1}
A^2(A/B)=\frac{3^{2n}+3^3}{-4}. 
\end{equation} For values of $n\in \{0,1,2\}$, it is then straightforward to determine all solutions of equation \eqref{equation1}, which are given in Table \ref{T4}. 
\begin{table}[h]
 \begin{center}
\begin{tabular}{c|ccc}
  $n$ &  0 & 1 & 2\\ [.25em]
 Solutions & none & $(A,B)\in \{(-3,3), (3,-3)\}$ & $(A,B)=(-3,1)$.
  \end{tabular}
\end{center}
\caption{Solutions to Equation \eqref{equation1} for $n\in \{0,1,2\}$}
 \label{T4}
\end{table} 
  
For $n\ge 3$, we rewrite equation \eqref{equation1} as
\begin{equation}\label{equation2}
  A^2(A/B)=3^3\left(\frac{3^{2n-3}+1}{-4}\right),
\end{equation} from which we deduce that $3\mid \mid A$ and $3\nmid B$.  Consequently, $3B\mid A$ and, rewriting equation \eqref{equation2}, we get 
\begin{equation}\label{equation3}
  (2B)^2\left(\frac{-A}{3B}\right)^3=3^{2n-3}+1 \equiv 1 \pmod{9}, 
\end{equation} 
since $n\ge 3$.
Since $f(x)$ is monogenic, it follows from condition \eqref{C2} of Theorem \ref{Thm:General} that 
\begin{align}\label{A mod 9, B mod 9}
\begin{split}
 (A\mmod{9}, \ B\mmod{9})&\in \{(3,1),(3,2),(3,8),(3,4),\\
 & \qquad (6,1),(6,5),(6,8),(6,7)\}. 
 \end{split}
\end{align} However, a simple computation reveals that 
\[(2B)^2\left(\frac{-A}{3B}\right)^3\mmod{9}\in \{4,5,7,8\},\] for each of the pairs $(A \mmod{9},\ B\mmod{9})$ in \eqref{A mod 9, B mod 9}, contradicting \eqref{equation3}. This contradiction shows that no  monogenic trinomials $f(x)$ with $\Gal(f)\simeq C_2\times A_4$ exist when $k=2$ and $n\ge 3$, which completes the proof in this direction.

Conversely, assume that conditions \eqref{Gen conditions C2 X A4} hold. When $k=2$, the fact that $f(x)$, where $(A,B)\in \{(-3,1),(-3,3),(3,-3)\}$, is monogenic with $\Gal(f)\simeq C_2 \times A_4$ can be confirmed easily from conditions \eqref{Gen conditions C2 X A4} and Theorem \ref{Thm:General}, or simply using a computer algebra system on the three trinomials,  and we omit the details. 

So, suppose that $k=1$ and let $f(x)\in \FF_1$. We show first that $\Gal(f)\simeq C_2 \times A_4$ by verifying that conditions \eqref{AJ conditions C2 X A4} hold. Observe that $-B$ cannot be a square, since $B$ is squarefree with $B\ne -1$. Thus, the first condition of \eqref{AJ conditions C2 X A4} is satisfied for $f(x)$.

We show next that the second condition in \eqref{AJ conditions C2 X A4} holds; that is, that $\Delta(g)$ is a square. If $3\mid A$, then 
\[(A \mmod{9}, \ B \mmod{9})\in \{(6,1),(6,4),(6,5),(6,8)\}.\] Hence, in this case we have 
\[\Delta(g)/(-3^3)=\delta/3^3=4(A/3)^3+B^2\equiv \left\{\begin{array}{cl}
 6 & \mbox{if $B\mmod{9}\in \{1,8\}$}\\
 3 & \mbox{if $B\mmod{9}\in \{4,5\}$,}
\end{array}\right.\] so that $3^4\mid \mid \Delta(g)$. Observe that if $3\nmid A$, then $3\nmid \Delta(g)$. Suppose next, regardless of whether $3\mid A$, that $p\ge 5$ is a prime divisor of $\Delta(g)=-\delta$. Since $\rad(\abs{\delta})\mid A$, 
it follows that $p\mid A$ and $p\mid B$. Then $p^2\mid\mid B^2$ since $B$ is squarefree, and therefore, 
\[\frac{\delta}{p^2}=\frac{4A^3+27B^2}{p^2}=4A(A/p)^2+27(B/p)^2\equiv 27(B/p)^2 \not \equiv 0 \pmod{p}.\] Thus, $p^2\mid \mid \Delta(g)$. Consequently,      
\begin{equation}\label{omega}
\abs{\Delta(g)}=\abs{\delta}=\abs{4A^3+27B^2}=\left\{\begin{array}{cl}
  3^4\prod_ip_i^2 & \mbox{if $3\mid A$,}\\[.5em]
  \prod_ip_i^2 & \mbox{if $3\nmid A$,}
\end{array}\right.
\end{equation} where the $p_i$ are the distinct prime divisors of $\delta$ with $p_i\ge 5$. Since $2\nmid B$, then we see that $\delta\equiv 3 \pmod{4}$ so that $\delta$ cannot be a square,  and we deduce from \eqref{omega} that $\abs{\delta}=-\delta$. Hence,  
$\Delta(g)$ is a square, which confirms that  the second condition of \eqref{AJ conditions C2 X A4} is satisfied for $f(x)\in \FF_1$. 

To show that $h(x)=x^6-Ax^2-B^2$ is irreducible, which is the third condition of \eqref{AJ conditions C2 X A4}, we proceed by way of contradiction and assume that $h(x)$ is reducible. Then, with $M(x)$ as defined in \eqref{M}, we have that 
  \begin{equation}\label{M=0}
  M(\mu)=\mu^4-2A\mu^2-8B\mu+A^2=0
   \end{equation} for some $\mu\in \Z$, by Lemma \ref{Lem:HJSextic}. Observe from \eqref{M=0} that $2\nmid \mu$ since $2\nmid A$ from \eqref{Gen conditions C2 X A4}. Solving \eqref{M=0} for $A$, we get
\[A=\mu^2\pm 2\sqrt{2B\mu}\in \Z.\] Thus, $2B\mu$ is a nonzero square, which yields the contradiction that $2\mid B\mu$. Hence, $h(x)$ is irreducible and $\Gal(f)\simeq C_2 \times A_4$.
The fact that $f(x)$ is monogenic follows easily from conditions \eqref{Gen conditions C2 X A4} and Theorem \ref{Thm:General}, and we omit the details. 

\subsection*{{\bf The Case} $\mathbf{S_4^{+}}$}
Assume that $f(x)$ is monogenic with $\Gal(f)\simeq S_4^{+}$.
By Theorem \ref{Thm:AJTri}, we have:
\begin{equation}\label{AJ conditions S4+}
 -B \ \mbox{is a square,} \quad \Delta(g) \ \mbox{is not a square,} \quad \mbox{and} \quad h(x) \ \mbox{is irreducible.}
\end{equation}
Since $-B$ is a square and $B$ is squarefree by Corollary \ref{Cor:B squarefree}, it follows that $B=-1$. 
If $A=(-1)^k3$, then we see from Case $A_4$ that $\Gal(f)\simeq A_4$. Hence, $A\ne (-1)^k3$. If $4\mid A$, then $(A\mmod{4},\ B\mmod{4})=(0,3)$, which contradicts condition \eqref{C2} of Theorem \ref{Thm:General}. Similarly, if $9\mid A$, then $(A\mmod{9},\ B\mmod{9})=(0,8)\not \in \RR_k$, which also contradicts condition \eqref{C2} of Theorem \ref{Thm:General}. Hence, $4\nmid A$ and $9\nmid A$. The fact that $\delta/3^{\nu_3(\delta)}$ is squarefree follows from condition \eqref{C5} of Theorem \ref{Thm:General}.  
We conclude that $f(x)\in \FF_2$, which proves the first direction.

Conversely, suppose that $f(x)\in \FF_2$. We show first that $\Gal(f)\simeq S_4^{+}$ by verifying that conditions \eqref{AJ conditions S4+} hold. Clearly, $-B$ is a square since $B=-1$. If $\Delta(g)$ is a square, then 
\begin{equation}\label{elliptic}
Y^2=X^3-2^4 3^3,
\end{equation} 
where 
\[(X,Y)=\left((-1)^k4A, \ (-1)^k16\delta\right).\] Using Sage to find the integral points on the elliptic curve \eqref{elliptic} yields $X=12$, so that $A=(-1)^k3$, which gives $f(x)$ with $\Gal(f)\simeq A_4$. Hence, we conclude that $\Delta(g)$ is not a square. To see that $h(x)$ is irreducible, we assume that $h(x)$ is reducible and proceed by way of contradiction. Exactly the same argument used in the proof of Corollary \ref{Cor:ShortList} can be used here with $B=-1$ to achieve a contradiction. Thus, $\Gal(f)\simeq S_4^{+}$. The fact that $f(x)$ is monogenic can be deduced from Theorem \ref{Thm:General}, and we omit the details, since they are essentially outlined in the proof of the first direction.  
\subsection*{{\bf The Case} $\mathbf{S_4^{-}}$}
Assume that $f(x)$ is monogenic with $\Gal(f)\simeq S_4^{-}$.
By Theorem \ref{Thm:AJTri}, we have:
\begin{equation}\label{AJ conditions S4-}
\begin{gathered}
 \mbox{neither $-B$ nor $\Delta(g)$ is a square,} \quad -B\Delta(g) \ \mbox{is a square,}\\ \quad \mbox{and} \quad h(x) \ \mbox{is irreducible.}
\end{gathered}
\end{equation}
Recall that 
\begin{equation}\label{-BDelta(g)}
-B\Delta(g)=-B(-B^{k-1}(4A^3+27B^{3-k}))=B^k(4A^3+27B^{3-k}).
\end{equation}

Suppose first that $k=1$, so that 
\begin{equation}\label{-BDelta(g)k=1}
-B\Delta(g)=B(4A^3+27B^2).
\end{equation} Then, since $B$ is squarefree from Corollary \ref{Cor:B squarefree} and $-B\Delta(g)$ is a square from \eqref{AJ conditions S4-}, we see from \eqref{-BDelta(g)k=1} that $B\mid (4A^3+27B^2)$. Thus, 
\begin{equation}\label{divis}
A\equiv \left\{\begin{array}{ll}
  0 \pmod{B} & \mbox{if $2\nmid B$}\\[.5em]
  0 \pmod{(B/2)} & \mbox{if $2\mid B$.}
\end{array}\right.
\end{equation}

If $2\nmid B$, then we see from \eqref{-BDelta(g)k=1} and \eqref{divis} that   
 \[B^3(4B(A/B)^3+27)\] is a square. Since $B$ is squarefree, it follows that $B\mid (4B(A/B)^3+27)$, so that $B\mid 27$, and we deduce that $B\in \{\pm 1,\pm 3\}$.
 
 Similarly, if $2\mid B$, then $B\mid 2A$ from \eqref{divis}, and 
 \[B^3(A(2A/B)^2+27)\] is a square. Thus, $(B/2)\mid (A(2A/B)^2+27)$, which implies that $(B/2)\mid 27$, and consequently, $B\in \{\pm 2, \pm 6\}$.   

To determine the possible corresponding values of $A$ for each of these values of $B$, we use Sage to find all integral points on the elliptic curves:
\begin{align*}
 E_1^{\pm}: \quad Y^2&=X^3\pm 2^43^3 \quad \mbox{when $B=\pm 1$, with $X=\pm 4A$;}\\
 E_2^{\pm}: \quad Y^2&=X^3\pm 2^43^8 \quad \mbox{when $B=\pm 3$, with $X=\pm 12A$;}\\
 E_3^{\pm}: \quad Y^2&=X^3\pm 2^33^3 \quad \mbox{when $B=\pm 2$, with $X=\pm 2A$;}\\
 E_4^{\pm}: \quad Y^2&=X^3\pm 2^33^8 \quad \mbox{when $B=\pm 6$, with $X=\pm 6A$.}
\end{align*} 
The curves $E_1^{+}$ and $E_2^{-}$ contain no integral points. 
For the other curves, we provide the values of $X$ and the corresponding viable integer pairs $(A,B)$ in Table \ref{T5}. 
  \begin{table}[h]
 \begin{center}
\begin{tabular}{c|cccccc}
  Curve & $E_1^{-}$ & $E_2^{+}$ & $E_3^{+}$ & $E_3^{-}$ & $E_4^{+}$ & $E_4^{-}$\\ [.25em]
 $X$ &  12 & 0 & $-6$ & 6,10,33 & $-18$ & 54,1942\\ [.25em]
 $(A,B)$ & $(-3,-1)$ & $\varnothing$ & $(-3,2)$ & $(-3,-2)$,$(-5,-2)$ & $(-3,6)$  & $(-9,-6)$.
 \end{tabular}
\end{center}
\caption{Integral $X$-coordinates on $E_{i}^{\pm}$ and corresponding pairs $(A,B)$} 
 \label{T5}
\end{table}

\noindent
It is straightforward to verify that $(A,B)=(-9,-6)$ is the only pair in Table \ref{T5} such that $f(x)$ is monogenic with $\Gal(f)\simeq S_4^{-}$, which completes the proof when $k=1$.  

Suppose now that $k=2$. It follows from \eqref{AJ conditions S4-} and \eqref{-BDelta(g)} that 
\begin{equation}\label{-BDelta(g)k=2}
\delta=4A^3+27B \quad \mbox{is a square.} 
\end{equation} Using the same argument as given in Case $S_3$ when $k=2$, we arrive at the equation
\[B=3-4(A/3)^3,\] given in \eqref{z=2}. Observe that $3\mid A$ since $B\in \Z$. Also, $B\ne -1$ by \eqref{AJ conditions S4-}, $B$ is squarefree by Corollary \ref{Cor:B squarefree}, and $4\nmid A$ by condition \eqref{C2} of Theorem \ref{Thm:General}. 

Conversely, it is straightforward to verify that these conditions are sufficient for $f(x)$ to be monogenic with $\Gal(f)\simeq S_4^{-}$. 

\subsection*{{\bf The Case} $\mathbf{C_2 \times S_4}$} The necessity of these conditions is straightforward and we omit the details. To prove that these conditions are sufficient, we need to show that $-B\delta(g)$ is not a square and that $h(x)$ is irreducible. To see that $-B\delta(g)$ is not a square, we assume the contrary, and proceed using the arguments contained in the Case $S_3$. 
To prove that $h(x)$ is irreducible, recall from Lemma \ref{Lem:HJSextic} that the reducibility of $h(x)$ implies that $M(\mu)=0$ for some unique integer $\mu$. Then, we proceed by way of contradiction and apply the the arguments used in the Case $C_2\times S_3$ to complete the proof. 
\end{proof}

\section{The Proof of Corollary \ref{Cor:Main}}
\begin{proof}
  We first address item \eqref{I1:Cor}, and we begin with the family 
  \[\FF_1=\{f(x): 3\nmid A \ \mbox{or} \ (A \mmod{9}, \ B\mmod{9})\in \{(6,1),(6,4),(6,5),(6,8)\}\}.\] Recall, in this scenario we have $k=1$ such that 
  \begin{equation}\label{F1conditions}
  2\nmid AB,\quad B\ne -1 \ \mbox{and is squarefree,}\quad \rad(\delta)\mid A,
  \end{equation} where $f(x)=x^6+Ax^2+B\in \FF_1$ is monogenic with $\Gal(f)\simeq C_2\times A_4$. We provide two infinite single-parameter subfamilies $\T_1$ and $\T_2$ of $\FF_1$, such that all elements of each subfamily generate distinct sextic fields. We let $t$ denote an integer parameter for both examples.
  
  For our first example, define 
  \begin{align*}
  A&:=-(27t^2+27t+7), \quad  B:=(2t+1)A \quad \mbox{and}\\ 
 \T_1&:=\{f_t(x)=x^6+Ax^2+B: t\ge 0 \ \mbox{and} \ B\ \mbox{is squarefree}\}. 
 \end{align*}
  Let $f_t(x)\in \T_1$ and let $g_t(x)=x^3+Ax+B$. Note that $-A>1$ and $\gcd(2t+1,A)=1$. Then, since $B=(2t+1)A$ is squarefree, we have that $A$ is squarefree and so $f_t(x)$ is Eisenstein with respect to every prime divisor of $A$. Thus, $f_t(x)$ is irreducible over $\Q$. An easy computation shows that   
  \[\Delta(g_t)=-\delta=-(4A^3+27B^2)=A^2.\] Hence, $f_t(x)$ satisfies the conditions in \eqref{F1conditions}, and since $3\nmid A$, we see that $\T_1\subset \FF_1$. Since there exist infinitely many integers $t$ such that $(2t+1)(27t^2+27t+7)$ is squarefree \cite{BB}, it follows that $\T_1$ and consequently, $\FF_1$,  is infinite. We claim that the elements of $\T_1$ generate distinct sextic fields. For $i\in \{1,2\}$, let $f_{t_i}(x)\in \T_1$ with $t_1\ne t_2$, let  $f_{t_i}(\theta_i)=0$ and $K_i=\Q(\theta_i)$. Assume, by way of contradiction, that $K_1=K_2$. Then, $\Delta(K_1)=\Delta(K_2)$, and since $f_{t_i}(x)$ is monogenic, it follows from \eqref{Eq:Dis-Dis} that $\Delta(f_{t_1})=\Delta(f_{t_2})$. Calculating $\Delta(f_{t_i})$ using Corollary \eqref{Cor:Swan}, we get
  \[64(2t_1+1)(27t_1^2+27t_1+7)^5-64(2t_2+1)(27t_2^2+27t_2+7)^5=64(t_1-t_2)T=0,\]
     where $T>0$, which yields the contradiction that $t_1=t_2$. Thus, if $t_1\ne t_2$, then $\Delta(f_{t_1})\ne \Delta(f_{t_2})$, and the claim that each element of $\T_1$ generates a distinct sextic field follows from Corollary \ref{Cor:distinct}. 
      We point out that the argument used in this first example to show that the sextic fields are distinct is similar to the argument needed to establish the same claim in the examples that follow, and so we omit the details.
   
   For our second example, we define 
   \begin{align*}
  A&:=-3(9t^2-21t+13), \quad  B:=(6t-7)A \quad \mbox{and}\\ 
 \T_2&:=\{f_t(x)=x^6+Ax^2+B: t\ge 0 \ \mbox{and} \ B\ \mbox{is squarefree}\}. 
  \end{align*} Note, as in the example $\T_1$, $f_t(x)\in \T_2$ is irreducible over $\Q$ since $f_t(x)$ is Eisenstein with respect to every prime divisor of $A$, due to the fact that $A$ is squarefree and $\gcd(6t-7,3(9t^2-21t+13))=1$. Here, with $g_t(x)=x^3+Ax+B$, we have 
  \[\Delta(g_t)=-\delta=-(4A^3+27B^2)=9A^2,\] so that $\rad(\delta)\mid A$. Hence, $f_t(x)$ satisfies \eqref{F1conditions}. Since $(A \mmod{9}, \ B \mmod{9})=(6,1)$, it follows that $\T_2\subset \FF_1$. 
  
  We turn next to the family
  \[\FF_2=\{f(x): B=-1, \  A\ne (-1)^k3, \ 4\nmid A, \ 9\nmid A \ \mbox{and} \ \delta/3^{\nu_3(\delta)} \ \mbox{is squarefree}\},\] where $f(x)\in \FF_2$ is monogenic with $\Gal(f)\simeq S_4^{+}$. We focus on $k=1$ since that will suffice to establish what we claim. We let $s$ be an integer parameter, and define 
 \[
  \SSS:=\{f_s(x)=x^6+(6s+1)x^{2}-1: s\ge 0 \ \mbox{ and $\delta$ is squarefree}\}, 
 \] where $\delta=4(6s+1)^3+27$. It is easy to see that $\SSS \subset \FF_2$, and since there exist infinitely many positive integers $s$ such that $\delta$ is squarefree \cite{Erdos}, it follows that $\SSS$ is infinite.

 Finally, for item \eqref{I1:Cor}, we consider the family  
 \[\FF_3=\{f(x): 3\mid A, \ 4\nmid A, \ B\ne -1, \ B=3-4(A/3)^3 \ \mbox{is squarefree}\},\]
 where $f(x)=x^6+Ax^4+B$ is monogenic with $\Gal(f)\simeq S_4^{-}$. We let $u$ be an integer parameter, and define 
 \begin{align*}
  A&:=6u+3, \quad  B:=3-4(2u+1)^3 \quad \mbox{and}\\ 
 \U&:=\{f_u(x)=x^6+Ax^4+B: u\ge 1 \ \mbox{and} \ B\ \mbox{is squarefree}\}. 
 \end{align*} Note that $B\ne -1$ since $u\ne 0$. It is clear that $\U\subset \FF_3$ and that $\U$ is infinite since there exist infinitely many positive integers $u$ such that $B$ is squarefree \cite{Erdos}. 
 
 We turn now to item \eqref{I2:Cor}. Let $v$ be an integer parameter. 
 
 Suppose that $k=1$, and define 
 \begin{align*}
  A&:=6v+1, \quad  B:=1 \quad \mbox{and}\\ 
 \V_1&:=\{f_v(x)=x^6+Ax^2+B: v\ge 0 \ \mbox{and} \ \delta\ \mbox{is squarefree}\}, 
 \end{align*} 
 where $\delta=4(6v+1)^3+27$. Note that $-\delta$ cannot be a square since $\delta>0$.  Then $f_v(x)\in \V_1$ is monogenic with $\Gal(f_v)\simeq C_2\times S_4$, and $\V_1$ is infinite since there exist infinitely many positive integers $v$ such that $\delta$ is squarefree \cite{Erdos}. 
 
 Suppose that $k=2$, and define 
 \begin{align*}
  A&:=36v+12, \quad  B:=1 \quad \mbox{and}\\ 
 \V_2&:=\{f_v(x)=x^6+Ax^4+B: v\ge 0 \ \mbox{and} \ \delta/\prod_{p\mid 2AB}p^{\nu_p(\delta)} \ \mbox{is squarefree}\}, 
 \end{align*} 
 where 
 \[\delta=4(36v+12)^3+27=3^3\left(2^8(3v+1)^3+1\right).\] Thus,
 \[\delta/\prod_{p\mid 2AB}p^{\nu_p(\delta)}=\delta/3^3=2^8(3v+1)^3+1,\] and since there exist infinitely many positive integers $v$ such that $2^8(3v+1)^3+1$ is squarefree \cite{Erdos}, we conclude that $\V_2$ is infinite. Since $(A \mmod{9}, \ B\mmod{9})=(3,1)$, it follows that $f_v(x)\in \V_2$ is monogenic with $\Gal(f_v)\simeq C_2\times S_4$.  
 \end{proof}
 
 \begin{rem}
   The infinite families given in the proof of Corollary \ref{Cor:Main} are not exhaustive, in the sense that other such single-parameter infinite families can be constructed in a similar manner. 
 \end{rem}








\begin{thebibliography}{99}

\bibitem{AJ} C. Awtrey and P. Jakes, \emph{Galois groups of even sextic polynomials}, Canad. Math. Bull.  {\bf 63} (2020), no. 3, 670--676.

\bibitem{AL} C. Awtrey and A. Lee, {\em Galois groups of reciprocal sextic polynomials}, Bull. Aust. Math. Soc. {\bf 109} (2024), 37--44.




\bibitem{BB}  A. Booker and T. D. Browning, \emph{Square-free values of reducible polynomials}, Discrete Anal. 2016, Paper No. 8, 16 pp.


\bibitem{BS} A. Bremner and B. Spearman, {\em Cyclic sextic trinomials $x^6+Ax+B$}, Int. J. Number Theory {\bf 6} (2010), no. 1, 161--167.


\bibitem{S1}  S. Brown, B. Spearman and Q. Yang, {\em On sextic trinomials with Galois group $C_6$, $S_3$ or $C_3\times S_3$}, J. Algebra Appl. {\bf 12} (2013), no. 1, 1250128, 9 pp.

\bibitem{S2} S. Brown, B. Spearman and Q. Yang, {\em On the Galois groups of sextic trinomials}, JP J. Algebra Number Theory Appl. {\bf 18} (2010), no. 1, 67--77.

\bibitem{BM} G. Butler and J. McKay, \emph{The transitive groups of degree up to eleven}, Comm.
Algebra {\bf 11} (1983), no. 8, 863--911.


\bibitem{Cohen} H. Cohen, \emph{A Course in Computational Algebraic Number Theory}, {Springer-Verlag}, 2000.


\bibitem{Erdos} P. Erd\H{o}s, {\em Arithmetical properties of polynomials}, J. London Math. Soc. {\bf 28} (1953), 416--425. 


\bibitem{G} K. Gajdzica, {\em Discriminants of special quadrinomials},
Rocky Mt. J. Math. {\bf 52} (2022), no. 5, 1587--1603.







\bibitem{HJMS} J. Harrington and L. Jones, \emph{The irreducibility of power compositional sextic polynomials and their Galois groups}, Math. Scand. {\bf 120} (2017), no. 2, 181--194.



\bibitem{HJBAMS} J. Harrington and L. Jones, {\em Monogenic quartic polynomials and their Galois groups}, Bull. Aust. Math. Soc. {\bf 111} (2025), no. 2, 244--259.

\bibitem{HJAA} J. Harrington and L. Jones, {\em Monogenic trinomials of the form $x^4+ax^3+d$ and their Galois groups}, J. Algebra Appl. (to appear).

\bibitem{HJActa} J. Harrington and L. Jones, {\em Monogenic sextic trinomials $x^6+ax^3+B$ and their Galois groups}, Acta Arith. (to appear).





\bibitem{JKS2} A. Jakhar, S. Khanduja and N. Sangwan, \emph{Characterization of primes dividing the index of a trinomial}, Int. J. Number Theory {\bf 13} (2017), no. 10, 2505--2514.



\bibitem{JonesRam} L. Jones, {\em Sextic reciprocal monogenic dihedral polynomials}, Ramanujan J. {\bf 56} (2021), no. 3, 1099--1110.

\bibitem{JonesNYJM} L. Jones, {\em Infinite families of reciprocal monogenic polynomials and their Galois groups}, New York J. Math. {\bf 27} (2021), 1465--1493.


\bibitem{JonesJAA}  L. Jones, {\em Monogenic reciprocal trinomials and their Galois groups}, J. Algebra Appl. 21 (2022), no. 2, Paper No. 2250026, 11 pp.


\bibitem{JonesQuarticsBAMS}  L. Jones, {\em Monogenic even quartic trinomials}, Bull. Aust. Math. Soc. {\bf 111} (2025), no. 2, 238--243.


\bibitem{JonesAA} L. Jones, \emph{Monogenic cyclic trinomials of the form $x^4+cx+d$}, Acta Arith. {\bf 218} (2025), no. 4, 385--394.

\bibitem{JonesRecipQuartics} L. Jones, {\em Monogenic reciprocal quartic polynomials and their Galois groups}, \url{ arXiv:2502.17691v1}.

\bibitem{JonesEvenSextics} L. Jones, {\em Monogenic even cyclic sextic polynomials}, Math. Slovaca (to appear). 

\bibitem{JW}  L. Jones and D. White, {\em Monogenic trinomials with non-squarefree discriminant}, Internat. J. Math. 32 (2021), no. 13, Paper No. 2150089, 21 pp.





\bibitem{MNSU} Y. Motoda, T. Nakahara, A.S.I. Shah and T. Uehara, \emph{On a problem of Hasse}, Algebraic number theory and related topics 2007, 209--221, RIMS K\^{o}ky\^{u}roku Bessatsu, B12, Res. Inst. Math. Sci. (RIMS), Kyoto, (2009).





\bibitem{Swan}  R. Swan,  \emph{Factorization of polynomials over finite fields}, Pacific J. Math. {\bf 12} (1962), 1099--1106.

\bibitem{Voutier} P. Voutier, {\em A family of cyclic quartic monogenic polynomials}, \url{arXiv:2405.20288v2}.


    

\end{thebibliography}

\end{document}